\documentclass[10pt,draft]{article}
\usepackage{amsmath,amscd}
\usepackage{amssymb,latexsym,amsthm}
\usepackage{color}
\usepackage[spanish,english]{babel}
\usepackage{amssymb}
\usepackage{color}
\usepackage{amsmath,amsthm,amscd}
\usepackage[latin1]{inputenc}
\usepackage{lscape}
\usepackage{fancyhdr}
\usepackage{amsfonts}
\usepackage{pb-diagram}

\numberwithin{equation}{section}
\newtheorem{theorem}{Theorem}[section]

\newtheorem{definition}[theorem]{Definition}
\newtheorem{proposition}[theorem]{Proposition}
\newtheorem{lemma}[theorem]{Lemma}

\newtheorem{corollary}[theorem]{Corollary}

\theoremstyle{definition}

\newtheorem{example}[theorem]{Example}

\newtheorem{remark}[theorem]{Remark}

\newcommand{\cO}{\mbox{${\cal O}$}}

\newcommand{\cU}{\mbox{${\cal U}$}}

\newcommand{\cW}{\mbox{${\cal W}$}}

\title{\textbf{Prime ideals of skew $PBW$ extensions}}
\author{Oswaldo Lezama\\
\texttt{jolezamas@unal.edu.co}
\\Juan Pablo Acosta \& Milton Armando Reyes Villamil
\\ Seminario de Álgebra Constructiva - SAC$^2$\\ Departamento de Matemáticas\\ Universidad Nacional de
Colombia, Sede Bogot\'a}
\date{}
\begin{document}
\maketitle
\begin{abstract}
\noindent In this paper we describe the prime ideals of some important classes of skew $PBW$
extensions, using the classical technique of extending and contracting ideals. Skew $PBW$
extensions include as particular examples Weyl algebras, enveloping algebras of finite-dimensional
Lie algebras (and its quantization), Artamonov quantum polynomials, diffusion algebras, Manin
algebra of quantum matrices, among many others.

\bigskip


\bigskip

\end{abstract}

\section{Introduction}

In this paper we describe the prime ideals of some important classes of skew $PBW$ extensions. For
this purpose we will consider the techniques that we found in \cite{Bell}, \cite{Goodearl} and
\cite{Goodearl2}. However, in several of our results, some modifications to these techniques should
be introduced. In this section we recall the definition of skew $PBW$ (Poincaré-Birkhoff-Witt)
extensions defined firstly in \cite{LezamaGallego}, and we will review also some elementary
properties about the polynomial interpretation of this kind of non-commutative rings. Weyl
algebras, enveloping algebras of finite-dimensional Lie algebras (and its quantization), Artamonov
quantum polynomials, diffusion algebras, Manin algebra of quantum matrices, are particular examples
of skew $PBW$ extensions (see \cite{lezamareyes1}).

\begin{definition}\label{gpbwextension}
Let $R$ and $A$ be rings. We say that $A$ is a \textit{skew $PBW$ extension of $R$} $($also called
a $\sigma-PBW$ extension of $R$$)$ if the following conditions hold:
\begin{enumerate}
\item[\rm (i)]$R\subseteq A$.
\item[\rm (ii)]There exist finite elements $x_1,\dots ,x_n\in A$ such $A$ is a left $R$-free module with basis
\begin{center}
${\rm Mon}(A):= \{x^{\alpha}=x_1^{\alpha_1}\cdots x_n^{\alpha_n}\mid \alpha=(\alpha_1,\dots
,\alpha_n)\in \mathbb{N}^n\}$.
\end{center}
In this case it says also that \textit{$A$ is a left polynomial ring over $R$} with respect to
$\{x_1,\dots,x_n\}$ and $Mon(A)$ is the set of standard monomials of $A$. Moreover, $x_1^0\cdots
x_n^0:=1\in Mon(A)$.
\item[\rm (iii)]For every $1\leq i\leq n$ and $r\in R-\{0\}$ there exists $c_{i,r}\in R-\{0\}$ such that
\begin{equation}\label{sigmadefinicion1}
x_ir-c_{i,r}x_i\in R.
\end{equation}
\item[\rm (iv)]For every $1\leq i,j\leq n$ there exists $c_{i,j}\in R-\{0\}$ such that
\begin{equation}\label{sigmadefinicion2}
x_jx_i-c_{i,j}x_ix_j\in R+Rx_1+\cdots +Rx_n.
\end{equation}
Under these conditions we will write $A:=\sigma(R)\langle x_1,\dots ,x_n\rangle$.
\end{enumerate}
\end{definition}
The following proposition justifies the notation and the alternative name given for the skew $PBW$
extensions.
\begin{proposition}\label{sigmadefinition}
Let $A$ be a skew $PBW$ extension of $R$. Then, for every $1\leq i\leq n$, there exists an
injective ring endomorphism $\sigma_i:R\rightarrow R$ and a $\sigma_i$-derivation
$\delta_i:R\rightarrow R$ such that
\begin{center}
$x_ir=\sigma_i(r)x_i+\delta_i(r)$,
\end{center}
for each $r\in R$.
\end{proposition}
\begin{proof}
See \cite{LezamaGallego}, Proposition 3.
\end{proof}

A particular case of skew $PBW$ extension is when all derivations $\delta_i$ are zero. Another
interesting case is when all $\sigma_i$ are bijective and the constants $c_{ij}$ are invertible. We
recall the following definition (cf. \cite{LezamaGallego}).
\begin{definition}\label{sigmapbwderivationtype}
Let $A$ be a skew $PBW$ extension.
\begin{enumerate}
\item[\rm (a)]
$A$ is quasi-commutative if the conditions {\rm(}iii{\rm)} and {\rm(}iv{\rm)} in Definition
\ref{gpbwextension} are replaced by
\begin{enumerate}
\item[\rm (iii')]For every $1\leq i\leq n$ and $r\in R-\{0\}$ there exists $c_{i,r}\in R-\{0\}$ such that
\begin{equation}
x_ir=c_{i,r}x_i.
\end{equation}
\item[\rm (iv')]For every $1\leq i,j\leq n$ there exists $c_{i,j}\in R-\{0\}$ such that
\begin{equation}
x_jx_i=c_{i,j}x_ix_j.
\end{equation}
\end{enumerate}
\item[\rm (b)]$A$ is bijective if $\sigma_i$ is bijective for
every $1\leq i\leq n$ and $c_{i,j}$ is invertible for any $1\leq i<j\leq n$.
\end{enumerate}
\end{definition}

Some extra notation will be used in the paper.

\begin{definition}\label{1.1.6}
Let $A$ be a skew $PBW$ extension of $R$ with endomorphisms $\sigma_i$, $1\leq i\leq n$, as in
Proposition \ref{sigmadefinition}.
\begin{enumerate}
\item[\rm (i)]For $\alpha=(\alpha_1,\dots,\alpha_n)\in \mathbb{N}^n$,
$\sigma^{\alpha}:=\sigma_1^{\alpha_1}\cdots \sigma_n^{\alpha_n}$,
$|\alpha|:=\alpha_1+\cdots+\alpha_n$. If $\beta=(\beta_1,\dots,\beta_n)\in \mathbb{N}^n$, then
$\alpha+\beta:=(\alpha_1+\beta_1,\dots,\alpha_n+\beta_n)$.
\item[\rm (ii)]For $X=x^{\alpha}\in {\rm Mon}(A)$,
$\exp(X):=\alpha$ and $\deg(X):=|\alpha|$.
\item[\rm (iii)]If $f=c_1X_1+\cdots +c_tX_t$,
with $X_i\in Mon(A)$ and $c_i\in R-\{0\}$, then $\deg(f):=\max\{\deg(X_i)\}_{i=1}^t.$
\end{enumerate}
\end{definition}
The skew $PBW$ extensions can be characterized in a similar way as was done in
\cite{Gomez-Torrecillas2} for $PBW$ rings.
\begin{theorem}\label{coefficientes}
Let $A$ be a left polynomial ring over $R$ w.r.t. $\{x_1,\dots,x_n\}$. $A$ is a skew $PBW$
extension of $R$ if and only if the following conditions hold:
\begin{enumerate}
\item[\rm (a)]For every $x^{\alpha}\in {\rm Mon}(A)$ and every $0\neq
r\in R$ there exist unique elements $r_{\alpha}:=\sigma^{\alpha}(r)\in R-\{0\}$ and $p_{\alpha
,r}\in A$ such that
\begin{equation}\label{611}
x^{\alpha}r=r_{\alpha}x^{\alpha}+p_{\alpha , r},
\end{equation}
where $p_{\alpha ,r}=0$ or $\deg(p_{\alpha ,r})<|\alpha|$ if $p_{\alpha , r}\neq 0$. Moreover, if
$r$ is left invertible, then $r_\alpha$ is left invertible.

\item[\rm (b)]For every $x^{\alpha},x^{\beta}\in {\rm Mon}(A)$ there
exist unique elements $c_{\alpha,\beta}\in R$ and $p_{\alpha,\beta}\in A$ such that
\begin{equation}\label{612}
x^{\alpha}x^{\beta}=c_{\alpha,\beta}x^{\alpha+\beta}+p_{\alpha,\beta},
\end{equation}
where $c_{\alpha,\beta}$ is left invertible, $p_{\alpha,\beta}=0$ or
$\deg(p_{\alpha,\beta})<|\alpha+\beta|$ if $p_{\alpha,\beta}\neq 0$.
\end{enumerate}
\begin{proof}
See \cite{LezamaGallego}, Theorem 7.
\end{proof}
\end{theorem}
We remember also the following facts from \cite{LezamaGallego}.
\begin{remark}\label{identities}
(i) We observe that if $A$ is quasi-commutative, then $p_{\alpha,r}=0$ and $p_{\alpha,\beta}=0$ for
every $0\neq r\in R$ and every $\alpha,\beta \in \mathbb{N}^n$.

(ii) If $A$ is bijective, then $c_{\alpha,\beta}$ is invertible for any $\alpha,\beta\in
\mathbb{N}^n$.

(iii) In $Mon(A)$ we define
\begin{center}
$x^{\alpha}\succeq x^{\beta}\Longleftrightarrow
\begin{cases}
x^{\alpha}=x^{\beta}\\
\text{or} & \\
x^{\alpha}\neq x^{\beta}\, \text{but} \, |\alpha|> |\beta| & \\
\text{or} & \\
x^{\alpha}\neq x^{\beta},|\alpha|=|\beta|\, \text{but $\exists$ $i$ with} &
\alpha_1=\beta_1,\dots,\alpha_{i-1}=\beta_{i-1},\alpha_i>\beta_i.
\end{cases}$
\end{center}
It is clear that this is a total order on $Mon(A)$. If $x^{\alpha}\succeq x^{\beta}$ but
$x^{\alpha}\neq x^{\beta}$, we write $x^{\alpha}\succ x^{\beta}$. Each element $f\in A$ can be
represented in a unique way as $f=c_1x^{\alpha_1}+\cdots +c_tx^{\alpha_t}$, with $c_i\in R-\{0\}$,
$1\leq i\leq t$, and $x^{\alpha_1}\succ \cdots \succ x^{\alpha_t}$. We say that $x^{\alpha_1}$ is
the \textit{leader monomial} of $f$ and we write $lm(f):=x^{\alpha_1}$ ; $c_1$ is the
\textit{leader coefficient} of $f$, $lc(f):=c_1$, and $c_1x^{\alpha_1}$ is the \textit{leader term}
of $f$ denoted by $lt(f):=c_1x^{\alpha_1}$.
\end{remark}
A natural and useful result that we will use later is the following property.
\begin{proposition}\label{1.1.10a}
Let A be a skew PBW extension of a ring R. If R is a domain, then A is a domain.
\end{proposition}
\begin{proof}
See \cite{lezamareyes1}.
\end{proof}

The next theorem characterizes the quasi-commutative skew $PBW$ extensions.

\begin{theorem}\label{1.3.3}
Let $A$ be a quasi-commutative skew $PBW$ extension of a ring $R$. Then,
\begin{enumerate}
\item[\rm (i)] $A$ is isomorphic to an iterated skew polynomial ring of
endomorphism type, i.e.,
\begin{center}
$A\cong R[z_1;\theta_1]\cdots [z_{n};\theta_n]$.
\end{center}
\item[\rm (ii)] If $A$ is bijective, then each
endomorphism $\theta_i$ is bijective, $1\leq i\leq n$.
\end{enumerate}
\end{theorem}
\begin{proof}
 See \cite{lezamareyes1}.
\end{proof}

\begin{theorem}\label{1.3.2}
Let $A$ be an arbitrary skew $PBW$ extension of $R$. Then, $A$ is a filtered ring with filtration
given by
\begin{equation}\label{eq1.3.1a}
F_m:=\begin{cases} R & {\rm if}\ \ m=0\\ \{f\in A\mid {\rm deg}(f)\le m\} & {\rm if}\ \ m\ge 1
\end{cases}
\end{equation}
and the corresponding graded ring $Gr(A)$ is a quasi-commutative skew $PBW$ extension of $R$.
Moreover, if $A$ is bijective, then $Gr(A)$ is a quasi-commutative bijective skew $PBW$ extension
of $R$.
\end{theorem}
\begin{proof}
See \cite{lezamareyes1}.
\end{proof}

\begin{theorem}[Hilbert Basis Theorem]\label{1.3.4}
Let $A$ be a bijective skew $PBW$ extension of $R$. If $R$ is a left $($right$)$ Noetherian ring
then $A$ is also a left $($right$)$ Noetherian ring.
\end{theorem}
\begin{proof}
 See \cite{lezamareyes1}.
\end{proof}

\section{Invariant ideals}

Let $A=\sigma(R)\langle x_1,\dots,x_n\rangle$ be a skew $PBW$ extension of a ring $R$. By
Proposition \ref{sigmadefinition}, we know that $x_ir-\sigma_i(r)x_i=\delta_i(r)$ for all $r\in R$,
where $\sigma_i$ is an injective endomorphism of $R$ and $\delta_i$ is a $\sigma_i$-derivation of
$R$, $1\leq i\leq n$. This motivates the following general definition.

\begin{definition}\label{1.4.1}
Let $R$ be a ring and $(\Sigma,\Delta)$ a system of endomorphisms and $\Sigma$-derivations of $R$,
with $\Sigma:=\{\sigma_1,\dots,\sigma_n\}$ and $\Delta:=\{\delta_1,\dots,\delta_n\}$.
\begin{enumerate}
\item[\rm (i)]If $I$ is an ideal of $R$, $I$ is called $\Sigma$-invariant if
$\sigma_i(I)\subseteq I$, for every $1\leq i\leq n$. In a similar way are defined the
$\Delta$-invariant ideals. If $I$ is both $\Sigma$ and $\Delta$-invariant, we say that $I$ is
$(\Sigma,\Delta)$-invariant.
\item[\rm (ii)]A proper $\Sigma$- invariant ideal $I$ of $R$ is
$\Sigma$-prime if whenever a product of two $\Sigma$-invariant ideals is contained in $I$, one of
the ideals is contained in $I$. $R$ is a $\Sigma$-prime ring if the ideal $0$ is $\Sigma$-prime. In
a similar way are defined the $\Delta$-prime and $(\Sigma,\Delta)$-prime ideals and rings.
\item[\rm (iii)]The system $\Sigma$ is commutative if $\sigma_i\sigma_j=\sigma_j\sigma_i$ for
every $1\leq i\leq n$. In a similar way is defined the commutativity for $\Delta$. The system
$(\Sigma,\Delta)$ is commutative if both $\Sigma$ and $\Delta$ are commutative.
\end{enumerate}
\end{definition}

The following proposition describe the behavior of these properties when we pass to a quotient
ring.

\begin{proposition}\label{1.4.2b}
Let $R$ be a ring, $(\Sigma,\Delta)$ a system of endomorphisms and $\Sigma$-derivations of $R$, $I$
a proper ideal of $R$ and $\overline{R}:=R/I$.
\begin{enumerate}
\item[\rm (i)]If $I$ is $(\Sigma,\Delta)$-invariant, then over $\overline{R}:=R/I$ is induced a system
$(\overline{\Sigma},\overline{\Delta})$ of endomorphisms and $\overline{\Sigma}$-\-de\-ri\-vations
defined by $\overline{\sigma_i}(\overline{r}):=\overline{\sigma_i(r)}$ and
$\overline{\delta_i}(\overline{r}):=\overline{\delta_i(r)}$, $1\leq i\leq n$. If $\sigma_i$ is
bijective and $\sigma_i(I)=I$, then $\overline{\sigma_i}$ is bijective.
\item[\rm (ii)]Let $I$ be $\Sigma$-invariant. If $\Sigma$ is commutative, then $\overline{\Sigma}$ is commutative. Similar
properties are valid for $\Delta$ and $(\Sigma,\Delta)$.
\item[\rm (iii)]Let $I$ be $\Sigma$-invariant. $I$ is $\Sigma$-prime if and only if $\overline{R}$ is $\overline{\Sigma}$-prime. Similar
properties are valid for $\Delta$ and $(\Sigma,\Delta)$.
\end{enumerate}
\end{proposition}
\begin{proof}
All statements follow directly from the definitions.
\end{proof}

According to the properties of $\Sigma$ and $\Delta$, we need to introduce some special classes of
skew $PBW$ extensions.

\begin{definition}
Let $A$ be a skew $PBW$ extension of a ring $R$ with system of endomorphisms
$\Sigma:=\{\sigma_1,\dots,\sigma_n\}$ and $\Sigma$-derivations
$\Delta:=\{\delta_1,\dots,\delta_n\}$.
\begin{enumerate}
\item[\rm (i)]If $\sigma_i=i_R$ for every $1\leq i\leq n$, we say that $A$ is a skew $PBW$ extension of
derivation type.
\item[\rm (ii)]If $\delta_i=0$ for every $1\leq i\leq n$, we say that $A$ is a skew $PBW$ extension of
endomorphism type. In addition, if every $\sigma_i$ is bijective, $A$ is a skew $PBW$ extension of
automorphism type.
\item[\rm (iii)]$A$ is $\Sigma$-commutative if the system $\Sigma$ is commutative.
In a similar way are defined the $\Delta$ and $(\Sigma,\Delta)$-commutativity of $A$.
\end{enumerate}
\end{definition}

Related with the previous definition, we have the following two interesting results. The second one
extends Lemma 1.5. (c) in \cite{Goodearl2}.

\begin{proposition}
Let $A$ be a skew $PBW$ extension of derivation type of a ring $R$. Then, for any
$\theta,\gamma,\beta\in \mathbb{N}^n$ and $c\in R$, the following identities hold:
\begin{center}
$c_{\gamma,\beta}c_{\theta,\gamma+\beta}=c_{\theta,\gamma}c_{\theta+\gamma,\beta}$,
$cc_{\theta,\gamma}=c_{\theta,\gamma}c$.
\end{center}
In particular, the system of constants $c_{i,j}$ are central.
\end{proposition}
\begin{proof}
This is a direct consequence of Remark \ref{identities}, part (ii).
\end{proof}

\begin{proposition}\label{1.4.9}
Let $A$ be a skew $PBW$ extension of a ring $R$. If for every $1\leq i\leq n$, $\delta_i$ is inner,
then $A$ is a skew $PBW$ extension of $R$ of endomorphism type.
\end{proposition}
\begin{proof}
Let $a_i\in R$ such that $\delta_i=\delta_{a_i}$ is inner, $1\leq i\leq n$. We will prove that
$A=\sigma(R)\langle z_1,\dots,z_n\rangle$, where $z_i:=x_i-a_i$, the system of endomorphisms
coincides with the original system $\Sigma$ and every $\sigma_i$-derivation is equal zero. We will
check the conditions in Definition \ref{gpbwextension}. It is clear that $R\subseteq A$. Let $r\in
R$, then
$z_ir=(x_i-a_i)r=x_ir-a_ir=\sigma_i(r)x_i+\delta_{a_i}(r)=\sigma_i(r)x_i+a_ir-\sigma_i(r)a_i-a_ir=\sigma_i(r)(x_i-a_i)=\sigma_i(r)z_i$.
Thus, the systems of constants $c_{i,r}$ of $\sigma(R)\langle z_1,\dots,z_n\rangle$ coincides with
the original one, and the same is true for the system of endomorphisms. Note that the system of
$\Sigma$-derivations is trivial, i.e., each one is equal zero. This means that $\sigma(R)\langle
z_1,\dots,z_n\rangle$ is of endomorphism type.

$z_jz_i=(x_j-a_j)(x_i-a_i)=x_jx_i-x_ja_i-a_jx_i+a_ja_i=c_{ij}x_ix_j+r_0+r_1x_1+\cdots+r_nx_n-x_ja_i-a_jx_i+a_ja_i$,
for some $r_0,r_1,\dots,r_n\in R$. Replacing $x_i$ by $z_i+a_i$  for every $1\leq i\leq n$, we
conclude that $z_jz_i-c_{ij}z_iz_j\in R+Rz_1+\cdots+Rz_n$.

Finally, note that $Mon\{z_1,\dots,z_n\}:=\{z^{\alpha}=z_1^{\alpha_1}\cdots
z_n^{\alpha_n}|\alpha=(\alpha_1,\dots ,\alpha_n)\in \mathbb{N}^n\}$ is a left $R$-basis of $A$. In
fact, it is clear that $Mon\{z_1,\dots,z_n\}$ generates $A$ as left $R$-module. Let $c_1,\dots,
c_t\in R$ such that $c_1z^{\alpha_1}+\cdots+c_tz^{\alpha_t}=0$ with $z^{\alpha_i}\in
Mon\{z_1,\dots,z_n\}$, $1\leq i\leq n$. Then, using the deglex order in Remark \ref{identities}, we
conclude that $c_1x^{\alpha_1}+\cdots+c_tx^{\alpha_t}$ should be zero, whence $c_1=\cdots=c_t=0$.
\end{proof}

In the next proposition we study quotients of skew $PBW$ extensions by $(\Sigma,\Delta)$-invariant
ideals.

\begin{proposition}\label{1.4.2a}
Let $A$ be a skew $PBW$ extension of a ring $R$ and $I$ a $(\Sigma,\Delta)$-invariant ideal of $R$.
Then,
\begin{enumerate}
\item[\rm (i)]$IA$ is an ideal of $A$ and $IA\cap
R=I$. $IA$ is proper if and only if $I$ is proper. Moreover, if for every $1\leq i\leq n$,
$\sigma_i$ is bijective and $\sigma_i(I)=I$, then $IA=AI$.
\item[\rm (ii)]If $I$ is proper and $\sigma_i(I)=I$ for every $1\leq i\leq n$, then $A/IA$ is a skew $PBW$ extension of
$R/I$. Moreover, if $A$ is of automorphism type, then $A/IA$ is of automorphism type. If $A$ is
bijective, then $A/IA$ is bijective. In addition, if $A$ is $\Sigma$-commutative, then $A/IA$ is
$\overline{\Sigma}$-commutative. Similar properties are true for the $\overline{\Delta}$ and
$(\overline{\Sigma},\overline{\Delta})$ commutativity.
\item[\rm (iii)]Let $A$ be of derivation type and $I$ proper. Then, $IA=AI$ and $A/IA$ is a skew $PBW$ extension of derivation type of
$R/I$.
\item[\rm (iv)]Let $R$ be left $($right$)$ Noetherian and $\sigma_i$ bijective for every $1\leq i\leq n$. Then, $\sigma_i(I)=I$ for every $i$ and $IA=AI$.
If $I$ is proper and $A$ is bijective, then $A/IA$ is a bijective skew $PBW$ extension of $R/I$.
\end{enumerate}
\end{proposition}
\begin{proof}
(i) It is clear that $IA$ is a right ideal, but since $I$ is $(\Sigma,\Delta)$-invariant, then $IA$
is also a left ideal of $A$. It is obvious that $IA\cap R=I$. From this last equality we get also
that $IA$ is proper if and only if $I$ is proper. Using again that $I$ is
$(\Sigma,\Delta)$-invariant, we get that $AI\subseteq IA$. Assuming that $\sigma_i$ is bijective
and $\sigma_i(I)=I$ for every $i$, then $IA\subseteq AI$.

(ii) According to (i), we only have to show that $\overline{A}:=A/IA$ is a skew $PBW$ extension of
$\overline{R}:=R/I$. For this we will verify the four conditions of Definition \ref{gpbwextension}.
It is clear that $\overline{R}\subseteq \overline{A}$. Moreover, $\overline{A}$ is a left
$\overline{R}$-module with generating set $Mon\{\overline{x_1},\dots,\overline{x_n}\}$. Next we
show that $Mon\{\overline{x_1},\dots,\overline{x_n}\}$ is independent. Consider the expression
$\overline{r_1}\ \overline{X_1}+ \dotsb + \overline{r_n}\overline{X_n}=\overline{0}$, where $X_i\in
{\rm Mon}(A)$ for each $i$. We have $r_1X_1+\dotsb + r_nX_n\in IA$ and hence
\[
r_1X_1+\dotsb + r_nX_n=r_1'X_1+\dotsb +r_n'X_n,\ \ \text{for some }\ r_i'\in I,\ i=1,\dotsc,n.
\]
Thus, $(r_1-r_1')X_1+\dotsb + (r_n-r_n')X_n=0$, so $r_i\in I$, i.e., $\overline{r_i}=\overline{0}$
for $i=1,\dotsc, n$.

Let $\overline{r}\neq \overline{0}$ with $r\in R$. Then $r\notin IA$, and hence, $r\notin I$, in
particular, $r\neq 0$ and there exists $c_{i,r}:=\sigma_i(r)\neq 0$ such that
$x_ir=c_{i,r}x_i+\delta_i(r)$. Thus,
$\overline{x_i}\,\overline{r}=\overline{c_{i,r}}\,\overline{x_i}+\overline{\delta_i(r)}$. Observe
that $\overline{c_{i,r}}\neq \overline{0}$, contrary $c_{i,r}=\sigma_i(r)\in IA\cap
R=I=\sigma_i(I)$, i.e, $r\in I$, a contradiction. This completes the proof of condition (iii) in
Definition \ref{gpbwextension}.

In $A$ we have $x_jx_i-c_{i,j}x_ix_j\in R+\sum_{t=1}^n Rx_t$, with $c_{i,j}\in R-\{0\}$, so in
$\overline{A}$ we get that $\overline{x_j}\, \overline{x_i}-\overline{c_{i,j}}\, \overline{x_i}\,
\overline{x_j}\in \overline{R}+\sum_{t=1}^n \overline{R}\, \overline{x_t}$. Since $I$ is proper and
$c_{i,j}$ is left invertible for $i<j$ and right invertible for $i>j$, then $\overline{c_{i,j}}\neq
\overline{0}$. This completes the proof of condition (iv) in Definition \ref{gpbwextension}.

By Proposition \ref{1.4.2b}, if $\sigma_i$ is bijective, then $\overline{\sigma_i}$ is bijective.
It is obvious that if every constant $c_{ij}$ is invertible, then $\overline{c_{ij}}$ is
invertible..

The statements about the commutativity follow from Proposition \ref{1.4.2b}.

(iii) This is direct consequence of (i) and (ii).

(iv) Considering the Noether condition and the ascending chain $I\subseteq
\sigma_{i}^{-1}(I)\subseteq \sigma_{i}^{-2}(I)\subseteq \cdots$ we get that $\sigma_i(I)=I$ for
every $i$. The rest follows from (i) and (ii).
\end{proof}

\section{Extensions of derivation type}

Now we pass to describe the prime ideals of skew $PBW$ extensions of derivation type. Two technical
propositions are needed first. The total order introduced in Remark \ref{identities} will be used
in what follows.

\begin{proposition}\label{Terminator}
Let $A$ be a skew $PBW$ extension of a ring $R$ such that $\sigma_i$ is bijective for every $1\leq
i\leq n $. Let $J$ be a nonzero ideal of $A$. If $f$ is a nonzero element of $J$ of minimal leader
monomial $x^{\alpha_t}$ and $\sigma^{\alpha_t}(r)=r$ for any $r\in {\rm rann}_R({\rm lc}(f))$, then
${\rm rann}_A(f)=({\rm rann}_R({\rm lc}(f)))A$.
\end{proposition}
\begin{proof}
Consider $0\neq f=m_1X_1+\dotsb +m_tX_t$ an element of $J$ of minimal leader monomial
$X_t=x^{\alpha_t}$, with $X_t\succ X_{t-1}\succ \cdots \succ X_1$. By definition of the right
annihilator, ${\rm rann}_R({\rm lc}(f)):=\{r\in R\mid m_tr=0\}$. From Theorem \ref{coefficientes}
we have
\[
fr=m_1X_1r+\dotsb + m_t(\sigma^{\alpha_t}(r)x^{\alpha_t}+p_{\alpha_t,r}),
\]
where $p_{\alpha_t,r}=0$ or $\deg(p_{\alpha_t,r})<|\alpha_t|$ if $p_{\alpha_t, r}\neq 0$. Note that
if $r\in {\rm rann}_R({\rm lc}(f))$, then $fr=0$. In fact, if the contrary is assumed, since
$\sigma^{\alpha_t}(r)=r$, we get $lm(fr)\prec X_t$ with $fr\in J$, but this is a contradiction
since $X_t$ is minimal. Thus, $f{\rm rann}_R({\rm lc}(f))=0$ and $f{\rm rann}_R({\rm lc}(f))A=0$.
Therefore ${\rm rann}_R({\rm lc}(f))A\subseteq {\rm rann}_A(f)$.

Next we will show that ${\rm rann}_A(f) \subseteq {\rm rann}_R({\rm lc}(f))A$. Let $u=r_1Y_1+\dotsb
+ r_kY_k$ an element of ${\rm rann}_A(f)$, then
\[
fu=(m_1X_1+\dotsb +m_tX_t)(r_1Y_1+\dotsb + r_kY_k)=0,
\]
which implies that $m_tX_tr_kY_k=0$, whence $m_t\sigma^{\alpha_t}(r_k)X_tY_k=0$, that is,
$m_t\sigma^{\alpha_t}(r_k)=0$ which means $\sigma^{\alpha_t}(r_k)\in {\rm rann}_R(m_t)$. Let
$\sigma^{\alpha_t}(r_k):=s$. Then $s\in {\rm rann}_R({\rm lc}(f))$. Note that
$r_k=\sigma^{-\alpha_t}(s)=s$; moreover $\sigma^{\alpha_t}(s)=s$ implies $s=\sigma^{-\alpha_t}(s)$
and hence $r_k\in {\rm rann}_R({\rm lc}(f))$. This shows that $r_kY_k\in {\rm rann}_R({\rm
lc}(f))A\subseteq {\rm rann}_A(f)$, but since $u\in {\rm rann}_A(f)$ then $u-r_kY_k\in {\rm
rann}_A(f)$. Continuing in this way we obtain that $r_{k-1}Y_{k-1}, r_{k-2}Y_{k-2},\dotsc,
r_1Y_1\in {\rm rann}_R({\rm lc}(f))A$, which guarantees that $u\in {\rm rann}_R({\rm lc}(f))A$.
Thus, we have proved that ${\rm rann}_A(f)\subseteq {\rm rann}_R({\rm lc}(f))A$.
\end{proof}

\begin{proposition}\label{Terminator1}
Let $A$ be a skew $PBW$ extension of derivation type of a ring $R$ and let $K$ be a non zero ideal
of $A$. Let $K'$ be the ideal of $R$ generated by all coefficients of all terms of all polynomials
of $K$. Then $K'$ is a $\Delta$-invariant ideal of $R$.
\end{proposition}
\begin{proof}
Let $k\in K'$, then $k$ is a finite sum of elements of the form $rcr'$, with $r,r'\in R$ and $c$ is
the coefficient of one term of some polynomial of $K$. It is enough to prove that for every $1\leq
i\leq n$, $\delta_i(rcr')\in K'$. We have
$\delta_i(rcr')=rc\delta_i(r')+r\delta_i(c)r'+\delta_i(r)cr'$. Note that
$rc\delta_i(r'),\delta_i(r)cr'\in K'$, so only rest to prove that $\delta_i(c)\in K'$. There exists
$p\in K$ such that $p=cx^{\alpha}+p'$, with $p'\in A$ and $x^{\alpha}$ does not appear in $p'$.
Note that the coefficients of all terms of $p'$ are also in $K'$. Observe that $x_ip\in K$ and we
have
\begin{center}
$x_ip=x_icx^{\alpha}+x_ip'=cx_ix^{\alpha}+\delta_i(c)x^{\alpha}+x_ip'$;
\end{center}
from the previous expression we conclude that the coefficient of $x^{\alpha}$ in $x_ip$ is
$\delta_i(c)+cr+r'$, where $r$ is the coefficient of $x^{\alpha}$ in $x_ix^{\alpha}$ and $r'$ is
the coefficient of $x^{\alpha}$ in $x_ip'$. Since $c\in K'$ we only have to prove that $r'\in K'$.
Let $p'=c_1x^{\beta_1}+\cdots+c_tx^{\beta_t}$, then $c_1,\dots,c_t\in K'$ and $x^{\alpha}\notin
\{x^{\beta_1},\dots, x^{\beta_t}\}$. We have
\begin{center}
$x_ip'=(c_1x_i+\delta_i(c_1))x^{\beta_1}+\cdots
+(c_tx_i+\delta_i(c_t))x^{\beta_t}=c_1x_ix^{\beta_1}+\delta_i(c_1)x^{\beta_1}+\cdots+c_tx_ix^{\beta_t}+\delta_i(c_t)x^{\beta_t}$;
\end{center}
from the previous expression we get that $r'$ has the form $r'=c_1r_1+\cdots+c_tr_t$, where $r_j$
is the coefficient of $x^{\alpha}$ in $x_ix^{\beta_j}$, $1\leq j\leq t$. This proves that $r'\in
K'$.
\end{proof}

The following theorem gives a description of prime ideals of skew $PBW$ extensions of derivation
type without assuming any conditions on the ring of coefficients. This result generalizes the
description of prime ideals of classical $PBW$ extensions given in Proposition 6.2 of \cite{Bell}.
Compare also with \cite{McConnell}, Proposition 14.2.5 and Corollary 14.2.6.

\begin{theorem}\label{Terminator2}
Let $A$ be a skew $PBW$ extension of derivation type of a ring $R$. Let $I$ be a proper
$\Delta$-invariant ideal of $R$. $I$ is a $\Delta$-prime ideal of $R$ if and only if $IA$ is a
prime ideal of $A$. In such case, $IA=AI$ and $IA\cap R=I$.
\begin{proof}
(i) By Proposition \ref{1.4.2a} we know that $A/IA$ is a skew $PBW$ extension of $R/I$ of
derivation type, $IA=AI$ and $IA\cap R=I$. Then we may assume that $I=0$. Note that if $R$ is not
$\Delta$-prime, then $A$ is not prime. Indeed, there exist $I,J\neq 0$ $\Delta$-invariant ideals of
$R$ such that $IJ=0$, so $IA,JA\neq 0$ and $IAJA=IJA=0$, i.e., $A$ is not prime.

Suppose that $R$ is $\Delta$-prime. We need to show that if $J,K$ are nonzero ideals of $A$, then
$JK\neq 0$. Let $K'$ as in Proposition \ref{Terminator1}, then $K'\neq 0$ and it is
$\Delta$-invariant. Now let $j$ be a nonzero element of $J$ of minimal leader monomial. If $jK=0$,
then taking $M=R$ and $T=J$ in Proposition \ref{Terminator} we get
\[
K\subseteq {\rm rann}_A(j)={\rm rann}_R({\rm lc}(j))A.
\]
Therefore ${\rm lc}(j)K'=0$ and hence ${\rm lann}_R(K')\neq 0$. We have ${\rm lann}_R(K')K'=0$, and
note that ${\rm lann}_R(K')$ is also $\Delta$-invariant. In fact, let $a\in {\rm lann}_R(K')$ and
$k'\in K'$, then $\delta_i(a)k'=\delta_i(ak')-a\delta_i(k')=\delta_i(0)-a\delta_i(k')=0$ since
$\delta_i(k')\in K'$. Thus, $R$ is not $\Delta$-prime, a contradiction. In this way $jK\neq 0$ and
so $JK\neq 0$, which concludes the proof.
\end{proof}
\end{theorem}

\section{Extensions of automorphism type}

In this section we consider the characterization of prime ideals for extensions of automorphism
type over commutative Noetherian rings.

\begin{proposition}\label{5.2.6N}
Let $A$ be a bijective skew $PBW$ extension of a ring $R$. Suppose that given $a,b\in R-\{0\}$
there exists $\theta\in \mathbb{N}^n$ such that either $aR\sigma^\theta(b)\neq 0$ or
$aR\delta^\theta(b)\neq 0$. Then, $A$ is a prime ring.
\end{proposition}
\begin{proof}
Suppose that $A$ is not a prime ring, then there exist nonzero ideals $I,J$ of $A$ such that
$IJ=0$. We can assume that $I:={\rm lann}_A(J)$ and $J:={\rm rann}_A(I)$. Let $u$ be a nonzero
element of $I$ with minimal leader monomial $x^{\alpha}$ and leader coefficient $c_u$. We will
prove first that $\sigma^{-\alpha}(c_u)\in I$, i.e., $\sigma^{-\alpha}(c_u)J=0$. Since ${\rm
rann}_A(I)\subseteq {\rm rann}_A(u)$, then it is enough to show that $\sigma^{-\alpha}(c_u){\rm
rann}_A(u)=0$. Suppose that $\sigma^{-\alpha}(c_u){\rm rann}_A(u)\neq 0$, let $v\in {\rm
rann}_A(u)$ of minimal leader monomial $x^{\beta}$ and leader coefficient $c_v$ such that
$\sigma^{-\alpha}(c_u)v\neq 0$. Since $uv=0$ and $c_{\alpha,\beta}$ is invertible (see Theorem
\ref{coefficientes} and Remark \ref{identities}), then $c_u\sigma^{\alpha}(c_v)=0$, whence
$lm(uc_v)\prec x^{\alpha}$. The minimality of $x^{\alpha}$ implies that $uc_v=0$, and hence
$u(v-c_vx^{\beta})=0$. Moreover, $v-c_vx^{\beta}\in {\rm rann}_A(u)$ and $lm((v-c_vx^{\beta})\prec
x^{\beta}$, so $\sigma^{-\alpha}(c_u)(v-c_vx^{\beta})=0$. However, $c_u\sigma^{\alpha}(c_v)=0$, so
we have $\sigma^{-\alpha}(c_u)c_v=0$ and hence $\sigma^{-\alpha}(c_u)v=0$, a contradiction.

Thus, $I\cap R\neq 0$, and by symmetry, $J\cap R\neq 0$. Let $0\neq a\in I\cap R$ and $0\neq b\in
J\cap R$, by the hypothesis there exists $\theta\in \mathbb{N}^n$ and $r\in R$ such that
$ar\sigma^\theta(b)\neq 0$ or $ar\delta^\theta(b)\neq 0$. If $\theta=(0,\dots,0)$, then $arb\neq
0$, and hence $IJ\neq 0$, a contradiction. If $\theta\neq (0,\dots,0)$, then
$arx^{\theta}b=ar(\sigma^{\theta}(b)x^{\theta}+p_{\theta,b})$, but note that the independent term
of $p_{\theta,b}$ is $\delta^{\theta}(b)$ (see Theorem \ref{coefficientes}, part (b)). Thus,
$arx^{\theta}b\neq 0$, i.e., $IJ\neq 0$, a contradiction.
\end{proof}

\begin{corollary}
If $R$ is a prime ring and $A$ is a bijective skew $PBW$ extension of $R$, then $A$ is prime and
$rad(A)=0$.
\end{corollary}
\begin{proof}
Since $R$ is prime and $A$ is bijective, given $a,b\in R-\{0\}$, then $aR\sigma^{\theta}(b)\neq 0$
for every $\theta\in \mathbb{N}^n$. Thus, from the previous proposition, $A$ is prime.
\end{proof}

\begin{lemma}\label{3.3.3b}
Let $A$ be a bijective $\Sigma$-commutative skew $PBW$ extension of automorphism type of a left
$(right)$ Noetherian ring $R$. Let $I$ be a proper ideal of $R$ $\Sigma$-invariant. $I$ is a
$\Sigma$-prime ideal of $R$ if and only if $IA$ is a prime ideal of $A$. In such case, $IA=AI$ and
$IA\cap R=I$.
\end{lemma}
\begin{proof}
By Proposition \ref{1.4.2a}, $IA=AI$ is a proper ideal of $A$, $I=IA\cap R$ and
$\overline{A}:=A/IA$ is a bijective skew $PBW$ extension of $\overline{R}:=R/I$ of automorphism
type. In addition, observe that $I$ es $\Sigma$-prime if and only if $\overline{R}$ is
$\overline{\Sigma}$-prime (Proposition \ref{1.4.2b}). Thus, we can assume that $I=0$, and hence, we
have to prove that $R$ is $\Sigma$-prime if and only if $A$ is prime.

$\Rightarrow)$: Suppose that $R$ is $\Sigma$-prime, i.e., $0$ is $\Sigma$-prime. According to
Proposition \ref{5.2.6N}, we have to show that given $a,b\in R-\{0\}$ there exists $\theta\in
\mathbb{N}^n$ such that $aR\sigma^\theta(b)\neq 0$. Let $L$ be the ideal generated by the elements
$\sigma^{\theta}(b)$, $\theta\in \mathbb{N}^n$; observe that $L\neq 0 $, and since $A$ is
$\Sigma$-commutative, $L$ is $\Sigma$-invariant. But $R$ is Noetherian and $A$ is bijective, then
$\sigma_i(L)=L$ for every $1\leq i\leq n$ (see Proposition \ref{1.4.2a}). This implies that
$Ann_R(L)$ is $\Sigma$-invariant, but $0$ is $\Sigma$-prime, therefore $Ann_R(L)=0$. Thus, $aL\neq
0$, so there exists $\theta\in \mathbb{N}^n$ such that $a\sigma^\theta(b)\neq 0$.

$\Leftarrow)$: Note that if $R$ is not $\Sigma$-prime, then $A$ is not prime. Indeed, there exist
$K,J\neq 0$ $\Sigma$-invariant ideals of $R$ such that $KJ=0$, so $KA,JA\neq 0$ and since $J$ is
$\Sigma$-invariant, then $AJ=JA$ and hence $KAJA=KJA=0$, i.e., $A$ is not prime.
\end{proof}

\begin{proposition}\label{1.4.21}
Any skew $PBW$ extension $A$ of a commutative ring $R$ is $\Sigma$-commutative.
\end{proposition}
\begin{proof}
For $i=j$ it is clear that $\sigma_j\sigma_i=\sigma_i\sigma_j$. Let $i\neq j$, say $i<j$, then for
any $r\in R$ we have $lc(x_jx_ir)=\sigma_j\sigma_i(r)c_{ij}=c_{ij}\sigma_i\sigma_j(r)$, but since
$R$ is commutative and $c_{ij}$ is invertible, then $\sigma_j\sigma_i(r)=\sigma_i\sigma_j(r)$.
\end{proof}

\begin{theorem}\label{3.3.5a}
Let $A$ be a bijective skew $PBW$ extension of automorphism type of a commutative Noetherian ring
$R$. Let $I$ be a proper $\Sigma$-invariant ideal of $R$. $I$ is a $\Sigma$-prime ideal of $R$ if
and only if $IA$ is a prime ideal of $A$. In such case, $IA=AI$ and $IA\cap R=I$.
\end{theorem}
\begin{proof}
This follows from Lemma \ref{3.3.3b} and Propositions \ref{1.4.2a} and \ref{1.4.21}.
\end{proof}

\section{Extensions of mixed type}

Our next task is to give a description of prime ideals of bijective skew $PBW$ extensions of mixed
type, i.e., when both systems $\Sigma$ and $\Delta$ could be non trivial. We will assume that the
ring $R$ is commutative, Noetherian and semiprime. The proof of the main theorem (Theorem
\ref{5.2.4N}) is an in Lemma \ref{3.3.3b}, but anyway we have to show first some preliminary
technical propositions.

\begin{definition}
Let $R$ be a commutative ring, $I$ a proper ideal of $R$ and $\overline{R}:=R/I$, we define
\begin{center}
$S(I):=\{a\in R|\overline{a}:=a+I \ \text{is regular }\}$.
\end{center}
\end{definition}
By regular element we mean a non zero divisor. Note that $S(0)$ is the set of regular elements of
$R$. Next we will describe the behavior of the properties introduced in Definition \ref{1.4.1} when
we pass to the total ring of fractions.

\begin{proposition}\label{1.4.12}
Let $R$ be a commutative ring with total ring of fractions $Q(R)$ and let $(\Sigma,\Delta)$ be a
system of automorphisms and $\Sigma$-derivations of $R$. Then,
\begin{enumerate}
\item[\rm (i)]Over $Q(R)$ is induced a system
$(\widetilde{\Sigma},\widetilde{\Delta})$ of automorphisms and $\widetilde{\Sigma}$-derivations
defined by $\widetilde{\sigma_i}(\frac{a}{s}):=\frac{\sigma_i(a)}{\sigma_i(s)}$ and
$\widetilde{\delta_i}(\frac{a}{s}):=-\frac{\delta_i(s)}{\sigma_i(s)}\frac{a}{s}+\frac{\delta_i(a)}{\sigma_i(s)}$.
\item[\rm (ii)]$Q(R)$ is $\widetilde{\Sigma}$-prime if and only if $R$ is $\Sigma$-prime. The same is
valid for $\Delta$ and $(\Sigma,\Delta)$.
\end{enumerate}
\end{proposition}
\begin{proof}
(i) This part can be proved not only for commutative rings but also in the noncommutative case (see
\cite{Lezama-OreGoldie}).

(ii) $\Rightarrow )$: Let $I,J$ be $\Sigma$-invariant ideals of $R$ such that $IJ=0$, then
$IJS(0)^{-1}=IS(0)^{-1}JS(0)^{-1}=0$, but note that $IS(0)^{-1},JS(0)^{-1}$ are
$\widetilde{\Sigma}$-invariant, so $IS(0)^{-1}=0$ or $JS(0)^{-1}=0$, i.e., $I=0$ or $J=0$.

$\Leftarrow )$: $K,L$ be two $\widetilde{\Sigma}$-invariant ideals of $Q(R)$ such that $KL=0$, then
$K=IS(0)^{-1}$ and $L=JS(0)^{-1}$, with $I:=\{a\in R|\frac{a}{1}\in K\}$ and $J:=\{b\in
R|\frac{b}{1}\in L\}$. Note that $IJ=0$, moreover, $I,J$ are $\Sigma$-invariant. In fact, if $a\in
I$, then $\frac{a}{1}\in K$ and hence $\widetilde{\sigma_i}(\frac{a}{1})=\frac{\sigma_i(a)}{1}\in
K$, i.e., $\sigma_i(a)\in I$, for every $1\leq i\leq n$. Analogously for $L$. Since $R$ is
$\Sigma$-prime, then $I=0$ or $J=0$, i.e., $K=0$ or $L=0$. The proofs for $\Delta$ and
$(\Sigma,\Delta)$ are analogous.

\end{proof}

Next we will use the following special notation: Let $m\geq 0$ be an integer, then $\sigma(m)$ will
denote the product of $m$ endomorphisms taken from $\Sigma$ in any order, and probably with
repetitions, i.e., $\sigma(m)=\sigma_{i_1}\cdots \sigma_{i_m}$, with $i_1,\dots, i_m\in \{1,\dots,
n\}$. For $m=0$ we will understand that this product is the identical isomorphism of $R$.

\begin{proposition}\label{1.4.14N}
Let $R$ be a commutative ring and $(\Sigma,\Delta)$ a system of endomorphisms and
$\Sigma$-derivations of $R$. Let $I$ be a $\Sigma$-invariant ideal of $R$. Set $I_0:=R$, $I_1:=I$,
and for $j\geq 2$,
\begin{center}
$I_j:=\{r\in I|\delta_{i_1}\sigma(m(1))\delta_{i_2}\sigma(m(2))\cdots
\delta_{i_l}\sigma(m(l))(r)\in I,\ \text{for all}\ l=1,\dots,j-1;\, i_1,\dots, i_l\in \{1,\dots,n\}
\text{and}\ m(1),\dots, m(l)\geq 0\}$.
\end{center}
Then,
\begin{enumerate}
\item[\rm (i)]$I_0\supseteq I_1\supseteq I_2\supseteq \cdots $.
\item[\rm (ii)]$\delta_i(I_j)\subseteq I_{j-1}$, for every $1\leq i\leq n$ and any $j\geq 1$.
\item[\rm (iii)]$I_j$ is a $\Sigma$-invariant ideal of $R$, for any $j\geq 0$.
\item[\rm (iv)]$II_j\subseteq I_{j+1}$, for any $j\geq 0$.
\end{enumerate}
\end{proposition}
\begin{proof}
(i) This is evident.

(ii) It is clear that for every $i$, $\delta_i(I_1)\subseteq I_0$. Let $j\geq 2$ and let $r\in
I_j$, then
\begin{center}
$\delta_{i_1}\sigma(m(1))\delta_{i_2}\sigma(m(2))\cdots \delta_{i_l}\sigma(m(l))(r)\in
I$, for all $l=1,\dots, j-1$.
\end{center}
From this obtain that
$\delta_{i_1}\sigma(m(1))\delta_{i_2}\sigma(m(2))\cdots
\delta_{i_l}\sigma(m(l))\delta_i\sigma(0)(r)\in I$ for all $l=1,\dots, j-2$. This means that
$\delta_i(r)\in I_{j-1}$.

(iii) It is clear that $I_0,I_1$ are $\Sigma$-invariant ideals. Let $j\geq 2$ and let $r\in I_j$,
then for every $\sigma_i$ we have
\begin{center}
$\delta_{i_1}\sigma(m(1))\delta_{i_2}\sigma(m(2))\cdots
\delta_{i_l}\sigma(m(l))(\sigma_i(r))=\delta_{i_1}\sigma(m(1))\delta_{i_2}\sigma(m(2))\cdots
\delta_{i_l}\sigma(m(l)+1)(r)\in I$.
\end{center}
This means that $\sigma_i(r)\in I_j$, i.e., $I_j$ is $\Sigma$-invariant. Let us prove that $I_j$ is
an ideal of $R$. By induction we assume that $I_j$ is an ideal. It is obvious that if $a,a'\in
I_{j+1}$, then $a+a'\in I_{j+1}$; let $r\in R$, then $a\in I_j$ and hence $ra\in I_j$, therefore
$\delta_{i_1}\sigma(m(1))\delta_{i_2}\sigma(m(2))\cdots \delta_{i_k}\sigma(m(k))(ra)\in I$ for all
$k<j$. Consider any $\sigma(m(j))$, we have $\sigma(m(j))(a)\in I_{j+1}$, so
$\sigma(m(j))(a),\delta_i(\sigma(m(j))(a))\in I_j$ for every $i$. Therefore,
\begin{center}
$\delta_i(\sigma(m(j))(ra))=\delta_i(\sigma(m(j))(r)\sigma(m(j))(a))=\sigma_i\sigma(m(j))(r)\delta_i(\sigma(m(j))(a))+
\delta_i(\sigma(m(j))(r))\sigma(m(j))(a)\in I_j$.
\end{center}
This implies that $\delta_{i_1}\sigma(m(1))\delta_{i_2}\sigma(m(2))\cdots
\delta_{i_j}\sigma(m(j))(ra)\in I$. This means that $ra\in I_{j+1}$. This proves that $I_{j+1}$ is
an ideal.

(iv) Of course $II_0\subseteq I_1$. Let $j\geq 1$, suppose that for $II_{j-1}\subseteq I_j$, we
have to prove that $II_j\subseteq I_{j+1}$. Let $a\in I$ and $b\in I_j$, then $b\in I_{j-1}$ and
$ab\in I_j$. Therefore, $\delta_{i_1}\sigma(m(1))\delta_{i_2}\sigma(m(2))\cdots
\delta_{i_k}\sigma(m(k))(ab)\in I$ for $i<j$. For any $\sigma(m(j))$ and every $\delta_i$ we have,
as above,
\begin{center}
$\delta_i(\sigma(m(j))(ab))=\delta_i(\sigma(m(j))(a)\sigma(m(j))(b))=\sigma_i\sigma(m(j))(a)\delta_i(\sigma(m(j))(b))+
\delta_i(\sigma(m(j))(a))\sigma(m(j))(b)\in II_{j-1}+RI_j\subseteq I_j$.
\end{center}
From this we conclude that $\delta_{i_1}\sigma(m(1))\delta_{i_2}\sigma(m(2))\cdots
\delta_{i_j}\sigma(m(j))(ab)\in I$, and this means that $ab\in I_{j+1}$. This completes the proof.
\end{proof}

\begin{proposition}\label{1.4.15}
Let $R$ be a commutative Noetherian ring and $\Sigma$ a system of automorphisms of $R$. Then, any
$\Sigma$-prime ideal of $R$ is semiprime.
\end{proposition}
\begin{proof}
Let $I$ be a $\Sigma$-prime ideal of $R$. Since $R$ is Noetherian, $\sqrt{I}$ is finitely
generated, so there exists $m\geq 1$ such that $(\sqrt{I})^m\subseteq I$. Since $I$ is
$\Sigma$-invariant and $R$ is Noetherian, $\sqrt{I}$ is $\Sigma$-invariant, whence
$\sqrt{I}\subseteq I$, i.e., $\sqrt{I}=I$. This means that $I$ is intersection of prime ideals,
i.e., $I$ is semiprime.
\end{proof}

\begin{proposition}\label{1.4.13aN}
Let $R$ be a commutative Noetherian ring and $(\Sigma,\Delta)$ a system of automorphisms and
$\Sigma$-derivations of $R$. Let $rad(R)$ be the prime radical of $R$. If $R$ is
$(\Sigma,\Delta)$-prime, then
\begin{enumerate}
\item[\rm (i)]$rad(R)$ is $\Sigma$-prime.
\item[\rm (ii)]$S(0)=S(rad(R))$.
\item[\rm (iii)]$Q(R)$ is Artinian.
\end{enumerate}
\end{proposition}
\begin{proof}
(i) The set of $\Sigma$-invariant ideals $I$ of $R$ such that $Ann_R(I)\neq 0$ is not empty since
$0$ satisfies these conditions. Since $R$ is assumed to be Noetherian, let $I$ be maximal with
these conditions. Let $K,L$ be $\Sigma$-invariant ideals of $R$ such that $I\varsubsetneq K$ and
$I\varsubsetneq L$, then $Ann_R(K)=0=Ann_R(L)$, and hence $Ann_R(KL)=0$. This implies that
$KL\nsubseteq I$. This proves that $I$ is $\Sigma$-prime.

We will prove that $I=rad(R)$. By Proposition \ref{1.4.14N}, we have the descending chain of
$\Sigma$-invariant ideals $I_0\supseteq I_1\supseteq \cdots$, and the ascending chain
$Ann_R(I_0)\subseteq Ann_R(I_1)\subseteq \cdots $. There exists $m\geq 1$ such that
$Ann_R(I_m)=Ann_R(I_{m+1})$ (since $I_0=R$ and $I_1=I$, see Proposition \ref{1.4.14N}, $m\neq 0$).
Note that $Ann_R(I_m)$ is $\Sigma$-invariant since $\sigma_i(I_m)=I_m$ for every $i$ (here we have
used again that $R$ is Noetherian, Proposition \ref{1.4.2a}). Let $b\in Ann_R(I_m)$. For $a\in
I_{m+1}$ we have $a\in I_m$, moreover, by Proposition \ref{1.4.14N}, for every $i$, $\delta_i(a)\in
I_m$, so $ab=0=\delta_i(a)b$, therefore $0=\delta_i(ab)=\sigma_i(a)\delta_i(b)$. From this we
obtain that $\sigma_i(I_{m+1})\delta_i(b)=0$, i.e., $I_{m+1}\delta_i(b)=0$. Thus, $\delta_i(b)\in
Ann_R(I_{m+1})=Ann_R(I_m)$. We have proved that $Ann_R(I_m)$ is $(\Sigma,\Delta)$-invariant.

Let $H:=Ann_R(Ann_R(I_m))$, we shall see that $H$ is also $(\Sigma,\Delta)$-invariant. In fact, let
$x\in H$, then $xAnn_R(I_m)=0$ and for every $i$ we have
$\sigma_i(x)\sigma_i(Ann_R(I_m))=0=\sigma_i(x)Ann_R(I_m)$, thus $\sigma_i(x)\in H$. Now let $y\in
Ann_R(I_m)$, then $xy=0$ and for every $i$ we have
$\delta_i(xy)=0=\sigma_i(x)\delta_i(y)+\delta_i(x)y$, but $\delta_i(y)\in Ann_R(I_m)$, so
$\sigma_i(x)\delta_i(y)=0$. Thus, $\delta_i(x)y=0$, i.e., $\delta_i(x)\in H$.

Since $R$ is $(\Sigma,\Delta)$-prime and $HAnn_R(I_m)=0$, then $H=0$ or $Ann_R(I_m)=0$. From
Proposition \ref{1.4.14N}, $Ann_R(I_m)\supseteq Ann_R(I)\neq 0$, whence $H=0$. Since $I_m\subseteq
H$, then $I_m=0$. Again, from Proposition \ref{1.4.14N} we obtain that $I^m\subseteq I_m$, so
$I^m=0$, and hence $I\subseteq rad(R)$. On the other hand, since $I$ is $\Sigma$-prime, then $I$ is
semiprime (Proposition \ref{1.4.15}), but $rad(R)$ is the smallest semiprime ideal of $R$, whence
$I=rad(R)$. Thus, $rad(R)$ is $\Sigma$-prime.

(ii) The inclusion $S(0)\subseteq S(rad(R))$ is well known (see \cite{McConnell}, Proposition
4.1.3). The other inclusion is equivalent to prove that $R$ is $S(rad(R))$-torsion free. Since
$I_m=0$, it is enough to prove that every factor $I_j/I_{j+1}$ is $S(rad(R))$-torsion free. In
fact, in general, if $M$ is an $R$-module and $N$ is an submodule of $M$ such that $M/N$ and $N$
are $S$-torsion free (with $S$ an arbitrary system of $R$), then $M$ is $S$-torsion free. Thus, the
assertion follows from
\begin{center}
$R=I_0/I_m$, $I_0/I_2/I_1/I_2\cong I_0/I_1$, $I_0/I_3/I_2/I_3\cong I_0/I_2$, $\dots,$
$I_0/I_m/I_{m-1}/I_m\cong I_0/I_{m-1}$.
\end{center}
$I_0/I_1=R/rad(R)$, it is clearly $S(rad(R))$-torsion free. By induction, we assume that
$I_{j-1}/I_j$ is $S(rad(R))$-torsion free. Let $a\in I_j$ and $r\in S(rad(R))$ such that
$r\widehat{a}=\widehat{0}$ in $I_j/I_{j+1}$, i.e., $ra\in I_{j+1}$. From Proposition \ref{1.4.14N}
we get that for every $i$ and any $\sigma(m)$, $\sigma(m)(a)\in I_j$, $\sigma(m)(ra)\in I_{j+1}$
and $\delta_i(\sigma(m)(ra))\in I_j$, then
\begin{center}
$\delta_i(\sigma(m)(ra))=\delta_i(\sigma(m)(r)\sigma(m)(a))=\sigma_i(\sigma(m)(r))\delta_i(\sigma(m)(a))
+\delta_i(\sigma(m)(r))\sigma(m)(a)\in I_j$,
\end{center}
whence, $\sigma_i(\sigma(m)(r))\delta_i(\sigma(m)(a))\in I_j$. For every $k$,
$\sigma_k(rad(R))=rad(R)$, then we can prove that $\sigma_k(S(rad(R)))=S(rad(R))$, so
$\sigma_i(\sigma(m)(r))\in S(rad(R))$; moreover, $\delta_i(\sigma(m)(a))\in I_{j-1}$, then by
induction $\delta_i(\sigma(m)(a))\in I_j$. But this is valid for any $i$ and any $\sigma(m)$, then
$a\in I_{j+1}$. This proves that $I_{j-1}/I_j$ is $S(rad(R))$-torsion free.

(iii) This follows from (ii) and Small's Theorem (see \cite{McConnell}, Corollary 4.1.4).
\end{proof}

\begin{corollary}\label{5.2.4aN}
Let $R$ be a commutative Noetherian semiprime ring and $(\Sigma,\Delta)$ a system of automorphisms
and $\Sigma$-derivations of $R$. If $R$ is $(\Sigma,\Delta)$-prime, then $R$ is $\Sigma$-prime.

\end{corollary}
\begin{proof}
By Proposition \ref{1.4.13aN}, the total ring of fractions $Q(R)$ of $R$ is Artinian. Then, by
Proposition \ref{1.4.12}, we can assume that $R$ is Artinian. Applying again Proposition
\ref{1.4.13aN}, we get that $0=rad(R)$ is $\Sigma$-prime, i.e., $R$ is $\Sigma$-prime.

\end{proof}

\begin{theorem}\label{5.2.4N}
Let $R$ be a commutative Noetherian semiprime ring and $A$ a bijective skew $PBW$ extension of $R$.
Let $I$ be a semiprime $(\Sigma,\Delta)$-invariant ideal of $R$. $I$ is a $(\Sigma,\Delta)$-prime
ideal of $R$ if and only if $IA$ is a prime ideal of $A$. In such case, $IA=AI$ and $I=IA\cap R$.
\end{theorem}
\begin{proof}
The proof is exactly as in Lemma \ref{3.3.3b}, anyway we will repeat it. By Proposition
\ref{1.4.2a}, $IA=AI$ is a proper ideal of $A$, $I=IA\cap R$ and $\overline{A}:=A/IA$ is a
bijective skew $PBW$ extension of the commutative Noetherian semiprime ring $\overline{R}:=R/I$. In
addition, observe that $I$ is $(\Sigma,\Delta)$-prime if and only if $\overline{R}$ is
$(\overline{\Sigma},\overline{\Delta})$-prime (see Propositions \ref{1.4.2b} and \ref{1.4.2a}).
Thus, we can assume that $I=0$, and hence, we have to prove that $R$ is $(\Sigma,\Delta)$-prime if
and only if $A$ is prime.

$\Rightarrow)$: From Corollary \ref{5.2.4aN}, we know that $R$ is $\Sigma$-prime, i.e., $0$ is
$\Sigma$-prime. Let $L$ be the ideal generated by the elements $\sigma^{\theta}(b)$, $\theta\in
\mathbb{N}^n$; observe that $L\neq 0 $, and since $A$ is $\Sigma$-commutative (Proposition
\ref{1.4.21}), $L$ is $\Sigma$-invariant. But $R$ is Noetherian and $A$ is bijective, then
$\sigma_i(L)=L$ for every $1\leq i\leq n$ (see Proposition \ref{1.4.2a}). This implies that
$Ann_R(L)$ is $\Sigma$-invariant, therefore $Ann_R(L)=0$. Thus, $aL\neq 0$, so there exists
$\theta\in \mathbb{N}^n$ such that $a\sigma^\theta(b)\neq 0$. From Proposition \ref{5.2.6N} we get
that $A$ is prime.

$\Leftarrow)$: Note that if $R$ is not $(\Sigma,\Delta)$-prime, then $A$ is not prime. Indeed,
there exist $I,J\neq 0$ $(\Sigma,\Delta)$-invariant ideals of $R$ such that $IJ=0$, so $IA,JA\neq
0$ and $IAJA=IJA=0$, i.e., $A$ is not prime.
\end{proof}

\section{Examples}

Some important examples of skew $PBW$ extensions covered by Theorems \ref{Terminator2},
\ref{3.3.5a} or \ref{5.2.4N} are given in the following table.  The definition of each ring or
algebra can be found in \cite{lezamareyes1}. For each example we have marked with $\checkmark$ the
theorem that can be applied.

\newpage

\begin{table}[htb]
\centering \tiny{
\begin{tabular}{|l|c|c|c|}\hline 
\textbf{Ring} & \textbf{\ref{Terminator2}} & \textbf{3.3.5} & \textbf{3.4.11}
\\ \hline \hline Habitual polynomial ring $R[x_1,\dotsc,x_n]$ & $\checkmark$ &  & \\
\cline{1-4} Ore extension of bijective type $R[x_1;\sigma_1 ,\delta_1]\cdots [x_n;\sigma_n
,\delta_n]$, &  & & \\
$R$ commutative Noetherian semiprime, $\delta_i\delta_j=\delta_j\delta_i$  &  & & $\checkmark$ \\
\cline{1-4} Weyl algebra $A_n(K)$ & $\checkmark$ &  & \\
\cline{1-4} Extended Weyl algebra $B_n(K)$ &
$\checkmark$ &  & \\
\cline{1-4} Universal enveloping algebra of a Lie algebra
$\mathcal{G}$, $\cU(\mathcal{G})$ & $\checkmark$ &  & \\
\cline{1-4} Tensor product $R\otimes_K \cU(\mathcal{G})$ & $\checkmark$ &  & \\
\cline{1-4} Crossed product $R*\cU(\mathcal{G})$ & $\checkmark$ &  & \\
\cline{1-4} Algebra of q-differential operators $D_{q,h}[x,y]$
&  &  & $\checkmark$ \\
\cline{1-4} Algebra of shift operators $S_h$ &  & $\checkmark$ & \\
\cline{1-4}
Mixed algebra $D_h$ &  &  & $\checkmark$ \\
\cline{1-4} Discrete linear systems $K[t_1,\dotsc,t_n][x_1,\sigma_1]\dotsb[x_n;\sigma_n]$ &  &
$\checkmark$ & \\
\cline{1-4} Linear partial shift operators
$K[t_1,\dotsc,t_n][E_1,\dotsc,E_n]$ & & $\checkmark$ & \\
\cline{1-4} Linear partial shift operators $K(t_1,\dotsc,t_n)[E_1,\dotsc,E_n]$ &  & $\checkmark$ & \\
\cline{1-4} L. P. Differential operators
$K[t_1,\dotsc,t_n][\partial_1,\dotsc,\partial_n]$ & $\checkmark$ &  & \\
\cline{1-4} L. P. Differential operators
$K(t_1,\dotsc,t_n)[\partial_1,\dotsc,\partial_n]$ & $\checkmark$ &  & \\
\cline{1-4} L. P. Difference operators
$K[t_1,\dotsc,t_n][\Delta_1,\dotsc,\Delta_n]$ &  &  & $\checkmark$ \\
\cline{1-4} L. P. Difference operators
$K(t_1,\dotsc,t_n)[\Delta_1,\dotsc,\Delta_n]$ &  &  & $\checkmark$ \\
\cline{1-4} L. P. $q$-dilation operators
$K[t_1,\dotsc,t_n][H_1^{(q)},\dotsc,H_m^{(q)}]$ &  & $\checkmark$ & \\
\cline{1-4} L. P. $q$-dilation operators
$K(t_1,\dotsc,t_n)[H_1^{(q)},\dotsc,H_m^{(q)}]$ &  & $\checkmark$ & \\
\cline{1-4} L. P. $q$-differential operators
$K[t_1,\dotsc,t_n][D_1^{(q)},\dotsc,D_m^{(q)}]$ &  &  & $\checkmark$\\
\cline{1-4} L. P. $q$-differential operators
$K(t_1,\dotsc,t_n)[D_1^{(q)},\dotsc,D_m^{(q)}]$ &  &  & $\checkmark$ \\
\cline{1-4} Diffusion algebras & $\checkmark$ &  & \\
\cline{1-4} Additive analogue of the Weyl algebra $A_n(q_1,\dotsc,q_n)$ &  &  & $\checkmark$ \\
\cline{1-4} Multiplicative analogue of the Weyl algebra
$\cO_n(\lambda_{ji})$ &  & $\checkmark$ & \\
\cline{1-4} Quantum algebra $\cU'(\mathfrak{so}(3,K))$ & $\checkmark$ &  & \\
\cline{1-4} 3-dimensional skew polynomial algebras & $\checkmark$ &  & \\
\cline{1-4} Dispin algebra $\cU(osp(1,2))$, & $\checkmark$ &  & \\
\cline{1-4} Woronowicz algebra $\cW_{\nu}(\mathfrak{sl}(2,K))$ & $\checkmark$
 &  & \\
\cline{1-4}
Complex algebra $V_q(\mathfrak{sl}_3(\mathbb{C}))$ &  & $\checkmark$ & \\
\cline{1-4} Algebra \textbf{U} &  & $\checkmark$ & \\
\cline{1-4} Manin algebra $\cO_q(M_2(K))$ &  & $\checkmark$ & \\
\cline{1-4} Coordinate algebra of the quantum group
$SL_q(2)$ &  & $\checkmark$ & \\
\cline{1-4} $q$-Heisenberg algebra \textbf{H}$_n(q)$ &  & $\checkmark$ & \\
\cline{1-4} Quantum enveloping algebra of
$\mathfrak{sl}(2,K)$, $\cU_q(\mathfrak{sl}(2,K))$ &  & $\checkmark$ & \\
\cline{1-4} Hayashi algebra $W_q(J)$ & $\checkmark$ &  & \\
\cline{1-4}
\end{tabular}}
\caption{Prime ideals of skew $PBW$ extensions.}
\end{table}

\begin{example}
Theorem \ref{3.3.5a} gives a description of prime ideals for the ring of skew quantum polynomials
over commutative Noetherian rings. Skew quantum polynomials were defined in \cite{lezamareyes1},
and represent a generalization of Artamonov's quantum polynomials (see \cite{Artamonov},
\cite{Artamonov2}). They can be defined as a localization of a quasi-commutative bijective skew
$PBW$ extension. We recall next its definition. Let $R$ be a ring with a fixed matrix of parameters
$\textbf{q}:=[q_{ij}]\in M_n(R)$, $n\geq 2$, such that $q_{ii}=1=q_{ij}q_{ji}=q_{ji}q_{ij}$ for
every $1\leq i,j\leq n$, and suppose also that it is given a system $\sigma_1,\dots,\sigma_n$ of
automorphisms of $R$. The ring of \textit{skew quantum polynomials over $R$}, denoted by
$R_{\textbf{q},\sigma}[x_1^{\pm 1 },\dots,x_r^{\pm 1}, x_{r+1},\dots,x_n]$, is defined as follows:
\begin{enumerate}
\item[\rm{(i)}]$R\subseteq R_{\textbf{q},\sigma}[x_1^{\pm 1},\dots,x_r^{\pm 1},
x_{r+1},\dots,x_n]$;
\item[\rm{(ii)}]$R_{\textbf{q},\sigma}[x_1^{\pm 1
},\dots,x_r^{\pm 1}, x_{r+1},\dots,x_n]$ is a free left $R$-module with basis
\begin{equation*}\label{equ1.4.2}
\{x_1^{\alpha_1}\cdots x_n^{\alpha_n}|\alpha_i\in \mathbb{Z} \ \text{for}\ 1\leq i\leq r \
\text{and} \ \alpha_i\in \mathbb{N}\ \text{for}\ r+1\leq i\leq n\};
\end{equation*}
\item[\rm{(iii)}] the variables $x_1,\dots,x_n$ satisfy the defining relations
\begin{center}
$x_ix_i^{-1}=1=x_i^{-1}x_i$, $1\leq i\leq r$,

$x_jx_i=q_{ij}x_ix_j$, $x_ir=\sigma_i(r)x_i$, $r\in R$, $1\leq i,j\leq n$.
\end{center}
\end{enumerate}
When all automorphisms are trivial, we write $R_{\textbf{q}}[x_1^{\pm 1 },\dots,x_r^{\pm 1},
x_{r+1},\dots,x_n]$, and this ring is called the ring of \textit{quantum polynomials over $R$}. If
$R=K$ is a field, then $K_{\textbf{q},\sigma}[x_1^{\pm 1 },\dots,x_r^{\pm 1}, x_{r+1},\dots,x_n]$
is the \textit{algebra of skew quantum polynomials}. For trivial automorphisms we get the
\textit{algebra of quantum polynomials} simply denoted by $\mathcal{O}_\textbf{q}$ (see
\cite{Artamonov}). When $r=0$, $R_{\textbf{q},\sigma}[x_1^{\pm 1},\dots,x_r^{\pm 1},
x_{r+1},\dots,x_n]=R_{\textbf{q},\sigma}[x_1,\dots,x_n]$ is the
\textit{$n$-mul\-ti\-pa\-ra\-me\-tric skew quantum space over $R$}, and when $r=n$, it coincides
with $R_{\textbf{q},\sigma}[x_1^{\pm 1},\dots,x_n^{\pm 1}]$, i.e., with the
\textit{$n$-multiparametric skew quantum torus over $R$}.

Note that $R_{\textbf{q},\sigma}[x_1^{\pm 1},\dots,x_r^{\pm 1}, x_{r+1},\dots,x_n]$ can be viewed
as a localization of the $n$-mul\-ti\-pa\-ra\-me\-tric skew quantum space, which, in turn, is a a
skew $PBW$ extension. In fact, we have the quasi-commutative bijective skew $PBW$ extension
\begin{equation*}\label{equ2.2.3}
A:=\sigma(R)\langle x_1,\dots, x_n\rangle, \, \text{with} \, x_ir=\sigma_i(r)x_i \, \text{and}\,
x_jx_i=q_{ij}x_ix_j, 1\leq i,j\leq n;
\end{equation*}
observe that $A=R_{\textbf{q},\sigma}[x_1,\dots,x_n]$. If we set
\begin{equation*}
S:=\{rx^{\alpha}\mid r\in R^*, x^{\alpha}\in {\rm Mon}\{x_1,\dots,x_r\}\},
\end{equation*}
then $S$ is a multiplicative subset of $A$ and
\begin{equation*}\label{equ2.7.4}
S^{-1}A\cong R_{\textbf{q},\sigma}[x_1^{\pm 1},\dots,x_r^{\pm 1}, x_{r+1},\dots,x_n]\cong AS^{-1}.
\end{equation*}
Thus, if $R$ is a commutative Noetherian ring, then Theorem \ref{3.3.5a} gives a description of
prime ideals for $A$. With this, we get a description of prime ideals for
$R_{\textbf{q},\sigma}[x_1^{\pm 1},\dots,x_r^{\pm 1}, x_{r+1},\dots,x_n]$ since it is well known
that there exists a bijective correspondence between the prime ideals of $S^{-1}A$ and the prime
ideals of $A$ with empty intersection with $S$ (recall that $A$ is left (right) Noetherian, Theorem
\ref{1.3.4}).
\end{example}



\end{document}